\documentclass[a4paper,11pt,reqno]{amsart}
\usepackage{epsfig,url,paralist}
\usepackage{amssymb,ulem,fullpage}
\usepackage{amsmath}
\usepackage{amsthm, tikz}
\usepackage{fullpage}
\usepackage{hyperref, cleveref,verbatim}
\usepackage{graphicx}
\usepackage{setspace}
\usepackage{enumitem}
\setlist{nolistsep}
\usepackage{lscape}

\numberwithin{equation}{section}

\newtheorem{theorem}{Theorem}
\newtheorem*{theorem*}{Main Theorem}
\newtheorem{remark}[theorem]{Remark}
\newtheorem{lem}[theorem]{Lemma}
\newtheorem{coro}[theorem]{Corollary}

\newtheorem{prop}[theorem]{Proposition}

\newtheorem{exm}[theorem]{Example}
\numberwithin{theorem}{section}
\newcommand{\tr}{\mathrm{tr}}

\newcommand{\slc}{{\rm SL_2}(\mathbb{C})}
\newcommand{\slk}{{\rm SL_2}(K)}
\newcommand{\slnk}{{\rm SL}_n(K)}

\newcommand{\pslk}{{\rm PSL_2}(K)}

\newcommand{\qp}{\mathbb{Q}_p}

\newcommand{\zp}{\mathbb{Z}_p}
\newcommand{\slqp}{{\rm SL_2}(\qp)}

\newcommand{\fix}{{\rm Fix}}
\newcommand{\ax}{{\rm Ax}}
\newcommand{\isomt}{{\rm Isom}(T)}

\newtheorem{thm}{Theorem}

\usepackage{fancyhdr}
\begin{document}
\title{Basic non-archimedean J{\o}rgensen theory}

\author{Matthew J. Conder, Harris Leung and Jeroen Schillewaert}
\email{\{matthew.conder,harris.pok.hei.leung,j.schillewaert\}@auckland.ac.nz}
\date{}
\maketitle

\begin{abstract}
We prove a non-archimedean analogue of J{\o}rgensen's inequality, and use it to deduce several algebraic convergence results. As an application we show that every dense subgroup of $\slqp$ contains two elements which generate a dense subgroup of $\slqp$, which is a special case of a result by Breuillard and Gelander. We also list several other related results, which are well-known to experts, but not easy to locate in the literature; for example, we show that a non-elementary subgroup of $\slk$ over a non-archimedean local field $K$ is discrete if and only if each of its two-generator subgroups is discrete.
\end{abstract}

\section{Introduction}

Let $K$ be a non-archimedean local field, that is, either a finite extension of the $p$-adic numbers $\qp$ or the field of formal Laurent series $\mathbb{F}_q((t))$. We study the group $\slk$ acting by isometries on the corresponding Bruhat-Tits tree $T_K$ \cite[Chapter II \S 1]{S}. We equip $\slk$ with the subspace topology inherited from $K^4$.

Our first result is a non-archimedean analogue of J{\o}rgensen's inequality \cite[Lemma~1]{J}. This generalises \cite[Theorem 4.2]{AP} and \cite[Theorem 1.5]{QYY} to discrete two-generator subgroups of $\slk$ which do not have a fixed end.

\begin{thm}\label{Jorg}
Let $K$ be a non-archimedean local field with discrete valuation $v$. There exists a non-negative constant $M_K$ such that, if $G=\langle A, B \rangle$ is a discrete subgroup of $\slk$ which does not fix an end of $T_K$, then
\begin{align}\label{JorgIneq} 
\min\{v(\tr^2(A)-4), v(\tr([A,B])-2) \} \le M_K.
\end{align}
\end{thm}

By a non-elementary subgroup of $\slk$, we mean a subgroup that does not stabilise a vertex, an end, or a pair of ends of $T_K$. In particular, \Cref{Jorg} holds when $G$ is non-elementary. 

Using results from \cite{CS}, we also obtain a more specialised version of \Cref{Jorg} for discrete and non-elementary two-generator subgroups of $\slk$. As is done in \cite{JKi} for the original statement of J{\o}rgensen's inequality, we additionally determine when equality occurs in this setting.

\begin{thm}\label{Jorg-AP}
Let $K$ be a non-archimedean local field with discrete valuation $v$. Let $G=\langle A, B \rangle$ be a discrete and non-elementary subgroup of $\slk$. If $K=\qp$, or $G$ contains no elements of order $p$ (where $p$ is the characteristic of the residue field of $K$), then
\begin{align}\label{JorgIneq0} 
\min\{v(\tr^2(A)-4), v(\tr([A,B])-2) \} \le 0.
\end{align}
Moreover, equality occurs if and only if $A$ is elliptic of finite order, $B$ is hyperbolic and the fixed point set of $A$ intersects the translation axis of $B$ in a finite path of length equal to the translation length of $B$.
\end{thm}

\begin{remark}
By \cite[Theorems A and B]{CS}, there are seven possible isomorphism classes for such a group $G$ which achieves equality in $(\ref{JorgIneq0})$. These isomorphism classes are identified in cases $(f)$ and $(g)$ of \cite[Theorem~A]{CS}. Furthermore, cases $(b)-(d)$ of \cite[Theorem A]{CS} describe the isomorphism classes for which the inequality $(\ref{JorgIneq0})$ is strict. 
\end{remark}

We use \Cref{Jorg} to obtain a non-archimedean analogue of \cite[Theorem 5.4.2]{Beardon}. Note that the assumption that the subgroup is non-elementary is necessary; see \Cref{elem-2d-exm}.

\begin{prop}\label{psl-prop-2d}
A non-elementary subgroup of $\slk$ is discrete if and only if each of its two-generator subgroups is discrete.
\end{prop}

In many cases, we obtain the following stronger result, which is also well-known to experts.

\begin{prop}\label{1d-criterion}
Let $G$ be a subgroup of $\slnk$, where $K$ is a non-archimedean local field. If either ${\rm char }(K)=0$, or ${\rm char }(K)=p>0$ and $G$ contains no elements of order $p$, then $G$ is discrete if and only if each of its cyclic subgroups are discrete.
\end{prop}

Let $\Gamma$ and $H$ be groups, and let $\{\phi_n:\Gamma \to G_n\}$ be a sequence of homomorphisms onto subgroups $G_n$ of $H$. If $\lim_{n \to \infty} \phi_n(\gamma)$ exists as an element of $H$ for each $\gamma \in \Gamma$, then we define the group $G = \{g \in H : g= \lim_{n \to \infty} \phi_n(\gamma),\,\gamma \in \Gamma\}$ and say that the sequence of groups $\{G_n=\phi_n(\Gamma)\}$ converges algebraically to $G$. We also say that a sequence $\{a_n\}$ is eventually X if there exists an $N$ such that the element $a_n$ is X for all $n\geq N$.

We use \Cref{Jorg} to deduce the following algebraic convergence results for subgroups of $\slk$. The first is a non-archimedean analogue of \cite[Theorem 1]{J}.

\begin{thm}\label{algconv1}
Let $\Gamma\le \slk$ be discrete and non-elementary and let $\{\phi_n:\Gamma\to G_n\}$ be a sequence of isomorphisms between $\Gamma$ and subgroups $G_n$ of $\slk$ which are eventually discrete. If $\{G_n\}$ converges algebraically to $G$, then $G$ is discrete and non-elementary, and the map $\phi: \Gamma \to G$ defined by $\phi(\gamma)=\lim_{n \to \infty} \phi_n(\gamma)$ is an isomorphism. 
\end{thm}
\begin{remark}
\Cref{algconv1} does not require $\Gamma$ to be finitely generated.
\end{remark}

We also give a non-archimedean analogue of the main theorem of \cite{JK}. To remain consistent with the terminology described above, when we say that a sequence $\{G_n\}$ of subgroups of $\slk$ converges algebraically to $G$, we will implicitly assume that there is an underlying group $\Gamma$ and surjective homomorphisms $\phi_n \colon \Gamma \to G_n$. For each element $g \in G$ (where $g= \lim_{n \to \infty} \phi_n(\gamma)\in G$ for some $\gamma \in \Gamma$), we will then use $g_n$ to denote the corresponding element $\phi_n(\gamma)$ of $G_n$, so that $g= \lim_{n \to \infty} g_n$ for each $g \in G$. In particular, by choosing $\Gamma$ to be the free group of rank $r$, this allows us to apply the notion of algebraic convergence to $r$-generator subgroups of $\slk$.

\begin{thm}\label{algconv2}
Let $\{G_n  = \langle g_{1_n},\dots,g_{r_n}\rangle\}$ be a sequence of discrete non-elementary $r$-generator subgroups of $\slk$, where $r$ is a positive integer. If $\{G_n\}$ converges algebraically to the group $G = \langle g_1,\cdots,g_r \rangle \le \slk$, then $G$
is discrete and non-elementary, and the maps $\psi_n:g_i\mapsto g_{i_n}$ extend to surjective homomorphisms $\psi_n:G\to G_n$ for sufficiently large $n$.
\end{thm}

In \Cref{alg-conv-exm}, we show that \Cref{algconv2} cannot be generalised to sequences of finitely generated subgroups of $\slk$ which are not discrete, or which are discrete but elementary.

\Cref{algconv2} implies the following non-archimedean analogue of the main proposition in \cite{J2}. Since the Lie group $\slqp$ is perfect, and its Lie algebra is generated by two elements \cite[Theorem 6]{Ku}, this is a very special case of a result by Breuillard and Gelander \cite[Corollary 2.5]{BG}.
\begin{prop}\label{dense}
Every dense subgroup of $\slqp$ contains two elements which generate a dense subgroup of $\slqp$. 
\end{prop}

\begin{remark}
Throughout the paper, we will often identify a subgroup $G$ of $\slk$ with its image $\overline{G}$ in $\pslk$, equipped with the quotient topology. By \cite[Proposition 2.4]{Kramer}, the topology on $\overline{G}$ is equivalent to the topology of pointwise convergence induced from the isometry group ${\rm Isom}(T_K)$. The latter is equivalent to the compact-open topology; see \S 2.4 Theorem~1 and \S 3.4 Definition 1 of \cite[Chapter X]{B}. Note that $G$ is discrete (respectively non-elementary) if and only if $\overline{G}$ is discrete (respectively non-elementary).
\end{remark}

\section{Proofs of \Cref{Jorg} and \Cref{Jorg-AP}}

We first establish the following criterion for discreteness (and its proof), which is a generalisation of \cite[Lemma 4.4.1]{K}.

\begin{lem}\label{disc-criterion}
Let $G$ be a topological group which acts by isometries on a locally finite simplicial complex $X$. Suppose that $G$ is equipped with the topology of pointwise convergence.
\begin{enumerate}[label={$(\arabic*)$}]
\item If $G$ is discrete, then ${\rm Stab}_G(y)$ is finite for every vertex $y\in X$.
\item If ${\rm Stab}_G(y)$ is finite for some vertex $y \in X$, then $G$ is discrete.
\end{enumerate}
\end{lem}
\begin{proof}
To prove $(1)$, suppose that $H={\rm Stab}_G(y)$ is infinite for some vertex $y$ of $X$. Let $B_n$ be the ball of radius $n \in \mathbb{N}$ about $y$.
Since $X$ is locally finite, the following inductive argument shows that there are infinitely many distinct elements of $H$ which fix each $B_n$ pointwise.

By the inductive hypothesis, we may assume that $H$ has infinitely many elements which fix $B_n$. Infinitely many of these elements, say $\{g_i : i\in I\}$, induce the same permutation on the finitely many vertices in $B_{n+1}\setminus B_n$. The set $\{g_ig_j^{-1} : i,j\in I\}$ then contains infinitely many distinct elements of $H$, each of which fixes $B_{n+1}$ pointwise. 
Hence we may choose a sequence $\{g_n\}$ of distinct non-trivial elements of $H$ such that $g_n$ fixes $B_n$ pointwise. Thus $\{g_n\}$ converges to 1, so $G$ is not discrete, which is a contradiction.

To prove $(2)$, suppose that ${\rm Stab}_G(y)$ is finite for some point $y$ and suppose that $\{g_i\}$ is a sequence of elements of $G$ converging to the identity. The sequence of vertices $\{g_i \cdot y \}$ converges to $y$, so there exists an $N$ such that $g_i \in {\rm Stab}_G(y)$ for each $i\geq N$. It follows that the sequence $\{g_i\}$ is eventually $1$, and hence $G$ is discrete.
\end{proof}

We now prove a non-archimedean version of \cite[Theorem 4.3.5(i)]{Beardon}.
 
\begin{lem}\label{trace-comm-2}
If $A \in \slk$ fixes an end $\xi \in \partial T_K$, then $B\in\slk$ fixes $\xi$ if and only if $\tr([A,B])=2$.
\end{lem}
\begin{proof}
We first note that the boundary $\partial T_K$ can be identified with the projective line $\mathbb{P}^1(K)$; see \cite[p. 72]{S}. Hence, after conjugation if necessary, we may assume that $\xi$ corresponds the eigenvector 
$(1,0)$ of $A$ or, equivalently, $A$ is upper triangular. A standard trace computation then shows that $\tr([A,B])=2$ if and only if $B$ is upper triangular (and hence also fixes $\xi$); see the proof of \cite[Theorem 4.3.5(i)]{Beardon} for details.
\end{proof}

\begin{remark}
If $A, B \in \slk$ are such that $\tr([A,B])=2$, then it is not necessarily the case that $A$ and $B$ must fix a common end of $T_K$; see $\Cref{tr-comm-2-ex}$.
\end{remark}

We also note the following two results.

\begin{prop}{\cite[Proposition II.3.15]{MS}}\label{TL}
The translation length of $X\in \slk$ on $T_K$ is
\begin{align*}
l(X)=-2\min\{0,v(\tr(X))\}.
\end{align*}
\end{prop}

\begin{lem}\label{finitely-many}
The set $\{\tr(X) : X \in \slk \textup{ has finite order} \}$ is finite.
\end{lem}
\begin{proof}
Let $q=p^r$ be the size of the residue field of $K$.
We prove that there are only finitely many possible orders (and hence traces) of finite order elements in $\slk$. Indeed, if an element $X\in \slk$ has finite order $n$ coprime to $p$, then \cite[Proposition 3.3]{CS} shows that $n \mid q \pm 1$. On the other hand, if $X$ has order $p^k$, then $X$ is unipotent and hence $k=1$ when ${\rm char}(K)>0$ by \cite[p. 964]{L}, and $(p-1)p^{k-1}\le [K:\qp]$ when ${\rm char}(K)=0$ by \cite[Proposition 17, p. 78]{S-loc}.
\end{proof}

We can now prove Theorems \ref{Jorg} and \ref{Jorg-AP}.

\begin{proof}[Proof of \Cref{Jorg}]
We first define the constant
$$M_K=\max\{v(\tr(X)-2) : X \in \slk \textup{ has finite order and } \tr(X) \neq 2\}.$$
\Cref{finitely-many} shows that $M_K$ is well-defined. Moreover, it is non-negative by \Cref{TL}.

Now suppose that $G=\langle A,B \rangle \le \slk$ is discrete and violates $(\ref{JorgIneq})$, that is, 
$$\min\{v(\tr^2(A)-4), v(\tr([A,B])-2) \} > M_K\ge 0.$$
Since $G$ contains no infinite order elliptic elements by Lemma $\ref{disc-criterion}$, it follows from \Cref{TL} that $\tr(A)=\pm 2$  and $\tr([A,B])=2$. Hence $A$ has repeated eigenvalue $\pm 1$ and fixes exactly one end $\eta$ of $T_K$. Lemma $\ref{trace-comm-2}$ then shows that $B$ fixes $\eta$, so $G$ fixes an end of $T_K$. 
\end{proof}

\begin{proof}[Proof of \Cref{Jorg-AP}]
Let $G=\langle A, B \rangle$ be a discrete non-elementary subgroup of $\slk$, where either $K=\qp$, or $G$ contains no elements of order $p$. Since $v(2)$ and $v(4)$ are non-negative, \Cref{TL} and the ultrametric inequality show that \Cref{JorgIneq0} holds with strict inequality if $A$ or $[A,B]$ is hyperbolic. Hence we may assume that $A$ and $[A,B]$ are elliptic. 

If $B$ is also elliptic, then the fixed point sets of $A$ and $B$ are disjoint since $G$ is non-elementary. By Lemma 2.2 of \cite[Chapter 3]{Chis}, the axes of $AB$ and $A^{-1}B^{-1}$ translate in the same direction along the unique geodesic between the fixed point sets of $A$ and $B$. It follows from Lemma 3.1 of \cite[Chapter 3]{Chis} that $[A,B]$ is hyperbolic, which is a contradiction. Hence $B$ is hyperbolic.

Note that $A$ and $BA^{-1}B^{-1}$ must fix a common vertex, as otherwise $[A,B]$ is hyperbolic by Lemma 2.2 of \cite[Chapter 3]{Chis}. Since $A$ cannot fix an end of the translation axis $\ax(B)$ of $B$, this implies that $A$ fixes a subpath of $\ax(B)$ of finite length $\Delta\ge l(B)$.

Since $G$ is discrete, $A$ has finite order and $\Delta=l(B)$ by \Cref{disc-criterion} and \cite[Lemma 3.10]{CS}. \Cref{TL} implies that $\Delta\ge 2$ and hence the fixed point set of $A$ cannot be a vertex or an edge. Thus $A$ fixes two ends of $T_K$ by \cite[Proposition 3.4]{CS} and is hence diagonalisable over $K$. Let $\lambda, \lambda^{-1} \in K$ be the eigenvalues of $A$ which, since $G$ is non-elementary, must be distinct $n$-th roots of unity for some $n>2$. By Hensel's Lemma, each $n$-th root of unity in $K$ is the unique lift of an $n$-th root of unity modulo $\pi$, where $\pi$ is the uniformiser of $K$. It follows that $\lambda \neq \lambda^{-1} \mod \pi$ and hence ${v(\tr^2(A)-4)}=2v(\lambda-\lambda^{-1})=0$,   so we obtain equality in (\ref{JorgIneq0}).
\end{proof}

\section{Non-elementary groups}

In the following two sections, we establish results that are necessary to prove the remaining statements in the introduction, some of which are interesting in their own right. Throughout, unless otherwise specified, we will use $T$ to denote a $\Lambda$-tree with path metric $d$, and $\isomt$ to denote the isometry group of $T$. We equip $\isomt$ with the topology of pointwise convergence. 

By replacing $T$ by an appropriate subdivision if necessary, we may always assume that every subgroup $G$ of $\isomt$ acts without inversions; see Lemma 1.3 of \cite[Chapter 3]{Chis}. Hence each element $g\in G$ is either elliptic or hyperbolic, and we denote the corresponding translation length by $l(g)$. We also denote the fixed point set of an elliptic isometry $g$ by $\fix(g)$, and the translation axis of a hyperbolic isometry $h$ by $\ax(h)$.

As in \cite[Section 3.1]{Gromov}, we say that a subgroup of $\isomt$ of $T$ is elementary if it stabilises a vertex, an end, or a pair of ends of $T$, and non-elementary otherwise.

\begin{lem}\label{distinct-axes}
If $G$ is a non-elementary subgroup of $\isomt$, then it contains two hyperbolic elements whose axes have pairwise distinct ends.
\end{lem}
\begin{proof}
We start by showing that $G$ must contain a hyperbolic element $g$. Suppose for a contradiction that all elements of $G$ are elliptic. Note that $\fix(g_i)\cap \fix(g_j) \neq \varnothing$ for each pair of distinct elements $g_i,g_j \in G$, as otherwise Lemma 2.2 of \cite[Chapter 3]{Chis} shows that $g_ig_j$ is hyperbolic. Hence $G$ fixes either a vertex or an end by \cite[Lemma 1.6]{Tits77}, which is a contradiction.

Let $\eta^+,\eta^-$ be the two ends of $\ax(g)$. We may suppose that the axis of every hyperbolic element of $G$ has an end in common with $\ax(g)$, as otherwise there is nothing to prove. Since $G$ is non-elementary, there are hyperbolic elements $h_1, h_2$ such that $h_1$ fixes $\eta^+$ (but not $\eta^-$) and $h_2$ fixes $\eta^-$ (but not $\eta^+$). If $\ax(h_1)$ and $\ax(h_2)$ have pairwise distinct ends, then the proof is complete, so we may suppose that $\ax(h_1)$ and $\ax(h_2)$ have a common end $\zeta \notin \{\eta^+, \eta^-\}$. Since the ends $\eta^+$ and $\zeta$ of $\ax(h_1)$ are distinct from the ends $h_2\cdot \eta^+$ and $\eta^-$ of $\ax(h_2gh_2^{-1})=h_2\cdot \ax(g)$, this proves the lemma.
\end{proof}

We now observe that, in a discrete group $G$ of isometries of a locally finite simplicial tree $T$, no element of $G$ can fix precisely one end of the axis of a hyperbolic element of $G$.

\begin{lem}\label{one-implies-two}
Let $G$ be a discrete subgroup of $\isomt$, where $T$ is a locally finite simplicial tree. Suppose that $h\in G$ fixes two ends $\eta$ and $\zeta$ of $T$. If $g\in G$ fixes at least one of $\eta$ or $\zeta$, then $g$ commutes with a power of $h$, whence $g$ fixes both $\eta$ and $\zeta$.
\end{lem}
\begin{proof}
Without loss of generality, we may suppose that $g$ fixes $\eta$. Let $P$ denote the unique bi-infinite path between $\eta$ and $\zeta$, and let us assume (by inverting $h$ if necessary) that if $h$ is hyperbolic then $\eta$ is the repelling fixed point of $\ax(h)$.

If $g$ is elliptic, then $g$ fixes a halfray $R$ on $P$. If $x$ is a point on $R$, then $h^igh^{-i}$ fixes $x$ for each $i\geq 0$. Since ${\rm Stab}_G(x)$ is finite by \Cref{disc-criterion}, we obtain that $h^igh^{-i}=h^jgh^{-j}$ for some distinct $i$ and $j$. Thus $g$ commutes with $h^k$, where $k=i-j$. 

If $g$ is hyperbolic, then $\ax(g)$ intersects $P$ in a halfray $R$ and we may assume (by inverting $g$ if necessary) that $\eta$ is the repelling fixed point of $\ax(g)$. Let $x \in R$ and observe that $gh^ig^{-1}h^{-i}$ fixes $x$ for each $i\geq 0$. A similar argument to above then shows that $g$ commutes $h^k$ for some integer $k$.

In either case, $g \cdot \zeta = gh^k \cdot \zeta = h^k (g \cdot \zeta)$, hence $g \cdot \zeta \in \{\eta,\zeta\}$. We conclude that $g$ fixes $\zeta$.
\end{proof}

We also prove the following non-archimedean analogue of \cite[Lemma 5]{JK}.

\begin{lem}\label{good-one}
Suppose that $G = \langle g_1,\cdots,g_r \rangle\le \isomt$ is a discrete non-elementary subgroup of $\isomt$, where $T$ is a locally finite simplicial tree. If $h \in G$ has infinite order, then the group $\langle g_i,h \rangle$ is non-elementary for some $i\in \{1,\cdots,r\}$.
\end{lem}
\begin{proof}
By \Cref{disc-criterion}, $h$ must be hyperbolic. Let $\eta^-$ and $\eta^+$ be the ends of $\ax(h)$ and suppose for a contradiction that $\langle g_i,h \rangle$ is elementary for every $i$. By \Cref{one-implies-two}, $g_i$ cannot fix precisely one of $\eta^-$ or $\eta^+$, and hence $g_i$ stabilises the set $\{\eta^-,\eta^+\}$. Thus $G$ stabilises $\{\eta^-,\eta^+\}$ and is hence elementary, which is a contradiction.
\end{proof}

We recall the following well-known non-archimedean analogue of Scott's core theorem \cite{Scott}.
\begin{lem}\label{finite-pres}
A discrete, non-elementary and compactly generated subgroup $G$ of $\slk$ is finitely presented.
\end{lem}
\begin{proof}
Since $G$ is a discrete subgroup of a Hausdorff group it is closed. Hence by Remark 2.3 and Lemma 2.4 of \cite{CD}, there exists a subtree of $T$ on which the action of $G$ is cocompact. Since $G$ acts properly on $T_K$, it follows from \cite[Corollary I.8.11]{BH} that $G$ is finitely presented. 
\end{proof}

Finally, we show that, within the class of discrete subgroups of $\slk$, the property of being non-elementary is preserved under isomorphism. In general, this property does not have to be preserved under isomorphism; see \Cref{non-elem-iso}.

\begin{lem}\label{non-elem-preserved}
If $G_1$ and $G_2$ are isomorphic discrete subgroups of $\slk$, then $G_1$ is non-elementary if and only if $G_2$ is non-elementary.
\end{lem}
\begin{proof}
We prove this by showing that a discrete subgroup of $\slk$ is elementary if and only if it is finite or has an abelian subgroup of index at most two.

Indeed, suppose that $G$ is a discrete elementary subgroup of $\slk$. If $G$ fixes a vertex of $T_K$, then $G$ is finite by \Cref{disc-criterion}. If $G$ preserves a pair of ends of $T_K$, then it contains a subgroup $H$ of index at most two which pointwise fixes a pair of ends of $T_K$. Since the boundary $\partial T_K$ can be identified with the projective line $\mathbb{P}^1(K)$ \cite[p. 72]{S}, we may assume by conjugation that $H$ consists entirely of diagonal matrices and is therefore abelian.
If $G$ fixes precisely one end $\eta$ of $T_K$, then by \Cref{one-implies-two}, every element of $G$ fixes $\eta$ and no other ends of $T$. Using the identification of $\partial T_K$ with $\mathbb{P}^1(K)$, we may then assume that $G$ consists entirely of upper triangular matrices with both eigenvalues either $1$ or $-1$. Thus $G$ is abelian.

Conversely, note that every finite subgroup of $\slk$ must fix a vertex of $T_K$ (and is therefore elementary) by Lemma 2.1 of \cite[Chapter 4]{Chis}. So suppose that $G$ is an infinite discrete subgroup of $\slk$ which contains an abelian subgroup $H$ of index at most two. We may assume that $H$ contains an infinite order element $h$, which must be hyperbolic by \Cref{disc-criterion}. Since $H$ is abelian, \Cref{trace-comm-2} shows that every element of $H$ fixes both ends $\eta^\pm$ of $\ax(h)$. Let $g\in G\setminus H$ and, since $H\trianglelefteq G$, we obtain that $(g^{-1}Hg)(\eta^{\pm}) = \eta^{\pm}$. It follows that $g$ stabilises $\{\eta^+, \eta^-\}$, hence $G$ is elementary.
\end{proof}

\section{Converging sequences}

Using the same notation as in the previous section, we continue to establish some results needed to prove the remaining statements in the introduction.

\begin{lem}\label{basic-convergence}
Let $\{g_n\}$ be a sequence of elements of $\isomt$ converging to $g \in \isomt$. 
\begin{enumerate}
\item[$(1)$] If $g$ is elliptic and there is a uniform lower bound $l_\mathrm{min}$ on the translation length of hyperbolic elements in the sequence $\{g_n\}$, then the sequence $\{g_n\}$ is eventually elliptic.
\item[$(2)$] If $g$ is hyperbolic, then the sequence $\{g_n\}$ is eventually hyperbolic.
\end{enumerate}
\end{lem}
\begin{proof}
If $g$ is elliptic, then let $x \in \fix(g)$. Using the topology of pointwise convergence, we may choose a positive integer $N$ such that $d(g_n\cdot x, x)=d(g_n\cdot x, g \cdot x)<l_\mathrm{min}$ for each $n \ge N$. It follows that $g_n$ fixes $x$ for each $n \ge N$.

On the other hand, if $g$ is hyperbolic, then assume for a contradiction that there is a subsequence $\{h_n\}$ of $\{g_n\}$ which consists only of elliptic elements and converges to $g$. Let $x$ be a point of $T$ and let $m$ be the midpoint of $[x, g\cdot x]$. Using the topology of pointwise convergence, we may choose a positive integer $N$ such that $d(h_n\cdot x,g \cdot x)<\frac{1}{2}l(g)$ and $d(h_n \cdot m, g \cdot m)<\frac{1}{2}l(g)$ for each $n \ge N$. By Lemma 1.1 of \cite[Chapter 3]{Chis}, the midpoint $m_n$ of $[x, h_n\cdot x]$ is fixed by $h_n$ for every $n$. Note that $d(m_n,m)\leq  \frac{1}{2} d(h_n \cdot x, g\cdot x)<\frac{1}{4}l(g)$ by \cite[Proposition II.2.2]{BH}  and thus 
$$d(h_n\cdot m,m)\leq d(h_n\cdot m,h_n\cdot m_n)+d(h_n \cdot m_n,m_n)+d(m_n,m)<\frac{1}{2}l(g)$$ 
for each $n \ge N$. Hence $d(g\cdot m,m)\leq d(g\cdot m,h_n\cdot m)+d(h_n\cdot m,m)<l(g)$ for sufficiently large $n$, which is a contradiction.
\end{proof}

\begin{remark}
Without the bound $l_{\min}$ in Lemma $\ref{basic-convergence}$ $(1)$, one could take a sequence of hyperbolic elements $\{g_n\}$ of translation length $\frac{1}{n}$ which converges to an elliptic element of $\isomt$.
\end{remark}

We obtain the following non-archimedean analogue of \cite[Lemma 2]{J} as a consequence of \Cref{basic-convergence} $(1)$. Note that this version does not use \Cref{Jorg}, contrasting with the proof in \cite{J} which requires J{\o}rgensen's inequality.

\begin{coro}\label{trace-constant}
Let $\{g_n\}$ be a sequence of elements of $\slk$ such that the cyclic groups $\langle g_n \rangle$ are discrete. If $\{g_n\}$ converges to an elliptic element $g \in \slk$, then the sequence $\{\tr(g_n)\}$ is eventually constant.
\end{coro}
\begin{proof}
By \Cref{basic-convergence} $(1)$, the sequence $\{g_n\}$ is eventually elliptic. It follows from \Cref{disc-criterion} that $\{g_n\}$ eventually consists of finite order elliptic elements. Since the trace function is continuous, \Cref{finitely-many} shows that $\{\tr(g_n)\}$ is eventually constant.
\end{proof}

\begin{lem}\label{elementary-preserve}
Let $\{G_n\}$ be a sequence of elementary subgroups of $\isomt$. If $\{G_n\}$ converges algebraically to $G$, 
then $G$ is elementary.
\end{lem}
\begin{proof}
Suppose for a contradiction that $G$ is non-elementary. By \Cref{distinct-axes}, there are hyperbolic elements $h_1, h_2 \in G$ such that the ends of $\ax(h_1)$ and $\ax(h_2)$ are distinct. Without loss of generality, we may assume that $h_1$ and $h_2$ translate in the same direction along the (possibly empty) finite path $\ax(h_1) \cap \ax(h_2)$. By Lemma 3.9 of \cite[Chapter 3]{Chis}, there is some positive integer $k$ such that $[h_1,h_2^k]$ is hyperbolic. By \Cref{basic-convergence} $(2)$, there is a sufficiently large positive integer $N$ such that the corresponding elements $h_{1_N}$, $h_{2_N}$ and $[h_{1_N},h_{2_N}^k]$ of $G_N$ are hyperbolic. Since $[h_{1_N},h_{2_N}^k]$ is hyperbolic, $\ax(h_{1_N})$ and $\ax(h_{2_N})$ have finite (or empty) overlap by the remark preceeding Proposition 3.7 of \cite[Chapter 3]{Chis}, which contradicts the fact that $G_N$ is elementary.
\end{proof}

The following is a non-archimedean analogue of \cite[Lemma 9]{JK}. Note that the assumption of eventual discreteness is necessary; see \Cref{approx-not-fix-end}.

\begin{lem}\label{approx-fixed-end}
Let $\{G_n=\langle g_n, h_n \rangle \}$ be a sequence of subgroups of $\slk$ converging algebraically to $G = \langle g,h \rangle \le \slk$, where $h$ is hyperbolic. If $\{G_n\}$ is eventually discrete, then $g$ fixes an end of $\ax(h)$ if and only if $g_n$ and $h_n$ fix a common end of $T_K$ for all sufficiently large $n$.
\end{lem}
\begin{proof}
Suppose first that $g_n$ and $h_n$ fix a common end of $T_K$ for all sufficiently large $n$. By \Cref{trace-comm-2}, $\tr([g_n,h_n])=2$ for all sufficiently large $n$, and so $\tr([g,h])=2$ since the trace function is continuous. Hence \Cref{trace-comm-2} shows that $g$ fixes both ends of $\ax(h)$.

Conversely, suppose that $g$ fixes an end of $\ax(h)$. By \Cref{trace-comm-2}, $\tr([g,h])=2$ and hence $[g,h]$ is elliptic by \Cref{TL}. \Cref{trace-constant} then shows that $\tr([g_n,h_n])=2$ for all sufficiently large $n$. Since $h_n$ is hyperbolic for sufficiently large $n$ by \Cref{basic-convergence} $(2)$, another application of \Cref{trace-comm-2} proves the result.
\end{proof}

We will use the following technical lemma several times.
\begin{lem}\label{technical}
Let $g \in \slk$ and suppose that $h_1, h_2 \in \slk$ are hyperbolic elements whose axes have pairwise distinct ends. If both $\langle g,h_1 \rangle$ and $ \langle g,h_2 \rangle$ are elementary groups, then $g=\pm 1$.
\end{lem}
\begin{proof}
Since $\langle g,h_1 \rangle$ and $\langle g,h_2 \rangle$ are elementary, $g$ stabilises the ends of both $\ax(h_1)$ and $\ax(h_2)$. Hence $g^2$ fixes four ends of $T_K$ and, since $\partial T_K$ can be identified with the projective line $\mathbb{P}^1(K)$ \cite[p. 72]{S}, it follows that $g^2=1$.
\end{proof}

We conclude this section by proving a non-archimedean analogue of \cite[Proposition 1]{J}.

\begin{prop}\label{disc-conv}
Let $G$ be a non-elementary subgroup of $\slk$ and let $\{G_n\}$ be a sequence of eventually discrete subgroups of $\slk$. If $\{G_n\}$ converges algebraically to $G$, then $G$ is discrete. 
\end{prop}
\begin{proof}
Let $\{g_i\}$ be a sequence of elements of $G$ converging to $1$. By Lemma \ref{distinct-axes}, we may choose hyperbolic elements $h_1, h_2 \in G$ whose axes have no end in common. Observe that the sequences $\{[g_i,h_1]\}$ and $\{[g_i,h_2]\}$ both converge to $1$. For sufficiently large $i$ and $n$, it follows that $\min\{v(\tr^2(g_{i_n})-4), \tr([g_{i_n},h_{j_n}])-2)\}>M_K$ for $j \in \{1,2\}$. \Cref{Jorg} hence implies that the subgroups $\langle g_{i_n}, h_{1_n} \rangle$ and $\langle g_{i_n}, h_{2_n}\rangle$ of $G_n$ are elementary for sufficiently large $i$ and $n$. By Lemma \ref{elementary-preserve}, the subgroups $\langle g_i, h_1 \rangle$ and $\langle g_i, h_2 \rangle$ of $G$ are also elementary for sufficiently large $i$. Hence $\{g_i\}$ is eventually constant by \Cref{technical}, so $G$ is discrete.
\end{proof}

\section{Proofs of the remaining statements}

\begin{proof}[Proof of \Cref{psl-prop-2d}]
Let $G$ be a non-elementary subgroup of $\slk$. If $G$ is discrete, then so is every subgroup of $G$. So suppose that every two-generator subgroup of $G$ is discrete, and that $\{g_i\}$ is a sequence of elements of $G$ converging to $1$. By \Cref{distinct-axes}, we may choose hyperbolic elements $h_1, h_2 \in G$ whose axes have pairwise distinct ends. For sufficiently large $i$, the following inequality holds for each $j \in \{1,2\}$:
\[ \min\{v(\tr^2(g_i)-4), v(\tr([g_i,h_j]-2)\}>M_K.\]
Hence \Cref{Jorg} shows that $\langle g_i, h_1\rangle$ and $\langle g_i, h_2\rangle$ are both elementary for sufficiently large $i$. By \Cref{technical}, $\{g_i\}$ is eventually constant and hence $G$ is discrete.
\end{proof}

\begin{proof}[Proof of \Cref{1d-criterion}]
If $G$ is discrete, then every cyclic subgroup of $G$ is discrete. Hence we may suppose that every cyclic subgroup of $G$ is discrete. Since $G$ acts by isometries on a locally finite simplicial building $X$ of type $\tilde{A}_{n-1}$ \cite[2.2.8 and 7.4.11]{BT72},  \Cref{disc-criterion} shows that every element in $G$ which fixes a vertex of $X$ has finite order. Thus every vertex stabiliser in $G$ is periodic, that is, it consists only of finite order elements. Moreover, the proof of \Cref{finitely-many} shows that every vertex stabiliser in $G$ has finite exponent.
Thus every vertex stabiliser in $G$ is finite by \cite[Theorem 9.1(ii) and (iii)]{Wehrfritz}, so $G$ is discrete by \Cref{disc-criterion}.
\end{proof}

\begin{proof}[Proof of \Cref{algconv1}]
We first prove that $\phi$ is an isomorphism. Since $\phi$ is surjective by construction, it suffices to show that $\phi$ is injective.

Let $A \in \Gamma$ be non-trivial. If $A$ is hyperbolic, then the proof of \Cref{distinct-axes} shows that we may choose a hyperbolic element $B \in \Gamma$ such that $\ax(A)$ and $\ax(B)$ have pairwise distinct ends. If $A$ is elliptic, then there exists a hyperbolic element $B \in \Gamma$ such that $A$ does not stabilise the set of ends of $\ax(B)$. In either case, there exists a hyperbolic element $B \in \Gamma$ such that $\langle A, B \rangle$ is non-elementary. By \Cref{non-elem-preserved}, $\langle \phi_n(A), \phi_n(B) \rangle$ is then a non-elementary subgroup of $G_n$ for sufficiently large $n$.

Now suppose for a contradiction that $\phi(A)=1$. The sequence $\{\phi_n(A)\}$ converges to $1$, so for sufficiently large $n$ we obtain $$\min\{v(\tr^2(\phi_n(A))-4),v(\tr([\phi_n(A),\phi_n(B)])-2)\}>M_K.$$ 
\Cref{Jorg} then implies that $\langle \phi_n(A), \phi_n(B) \rangle$ is elementary for sufficiently large $n$, which gives the desired contradiction. Hence $\phi$ is an isomorphism.

Now suppose that $\{g_i=\phi(\gamma_i)\}$ is a sequence of elements of $G$ converging to $1$. 
Since $G_n$ is isomorphic to $\Gamma$, \Cref{non-elem-preserved} shows that $G_n$ is non-elementary for sufficiently large $n$. For each such $n$, there exist hyperbolic elements $\phi_n(h_1), \phi_n(h_2) \in G_n$ whose axes have pairwise distinct ends by \Cref{distinct-axes}.
For sufficiently large $i$ and $n$, observe that 
$$\min\{v(\tr^2(\phi_n(\gamma_i)-4)), v(\tr([\phi_n(\gamma_i),\phi_n(h_j)])-2)\}>M_K$$ 
for $j \in \{1,2\}$. \Cref{Jorg} hence implies that the subgroups $\langle \phi_n(\gamma_i), \phi_n(h_1) \rangle$ and $\langle \phi_n(\gamma_i), \phi_n(h_2)\rangle$ of $G_n$ are elementary for sufficiently large $i$ and $n$.
\Cref{technical} thus shows that $\phi_n(\gamma_i)=\pm 1$ for sufficiently large $i$ and $n$. Since $\phi_n$ and $\phi$ are isomorphisms, it follows that $\{g_i=\phi(\gamma_i)\}$ is eventually 1, whence $G$ is discrete. Moreover, \Cref{non-elem-preserved} shows that $G$ is non-elementary.
\end{proof}

\begin{proof}[Proof of \Cref{algconv2}]
By \Cref{disc-conv}, if $G$ is non-elementary it is also discrete, so we start by proving the former.

For each $n$, one of the elements in the set $\{g_{i_n}, g_{i_n}g_{j_n} : 1\leq i<j \leq r\}$ must be hyperbolic as otherwise Lemma 2.1 of \cite[Chapter 4]{Chis} shows that $G_n$ fixes a point and is hence elementary. We will denote this hyperbolic element by $h_n$.
By Lemma \ref{good-one}, for each $n$ there exists $i_n \in \{1,\dots, r\}$ such that the group $H_n = \langle h_n, g_{i_n} \rangle$ is discrete and non-elementary. Since each $G_n$ is finitely generated, for infinitely many $n$ we must have picked the same indices to define both generators of $H_n$. Hence there exists a subsequence of $\{H_n\}$ converging to some two-generator subgroup $H$ of $G$. Since $G$ is non-elementary when $H$ is, it thus suffices to prove the statement for $r=2$. 

If every element of $\{g_1, g_2, g_1g_2\}$ is elliptic, then \Cref{basic-convergence} $(1)$ shows that every element of $\{g_{1_n}, g_{2_n}, g_{1_n}g_{2_n}\}$ is also elliptic for sufficiently large $n$. Lemma 2.1 of \cite[Chapter 4]{Chis} then shows that $G_n$ is elementary, which is a contradiction. Hence we may choose some hyperbolic element $h \in \{g_1, g_2, g_1g_2\}$. 

If $g \in \{g_1, g_2, g_1g_2\} \backslash \{h\}$ fixes an end of $\ax(h)$, then \Cref{approx-fixed-end} shows that $G_n$ is elementary for sufficiently large $n$, which is a contradiction. Therefore there is some $g \in \{g_1, g_2, g_1g_2\}\backslash\{h\}$ such that the hyperbolic elements $h$ and $ghg^{-1}$ have no ends in common, so $G$ is non-elementary (and hence discrete). 

To prove that the maps $\psi_n:g_i\mapsto g_{i_n}$ extend to surjective homomorphisms, we argue as in Theorem 2 of \cite{J}. Note that \Cref{finite-pres} shows that $G$ is finitely presented. Therefore a necessary and sufficient condition for $\psi_n$ to be a homomorphism is that the images of relators in $G$ are equal to the identity.

Let $R_1,\cdots,R_k$ be a basis for the relations for $G$. It remains to find a sufficiently large $N$ such that for $n\geq N$ the corresponding elements in $G_n$ are equal to the identity.  By Lemma \ref{distinct-axes}, we can find two hyperbolic elements $h_1, h_2 \in G$ whose axes have pairwise distinct ends. It follows from \Cref{basic-convergence} $(2)$ and \Cref{approx-fixed-end} that, for sufficiently large $n$, there exist hyperbolic elements $h_{1_n}, h_{2_n} \in G_n$ whose axes have pairwise distinct ends.

Since the sequence $\{G_n\}$ converges to $G$, the corresponding elements $R_{s_n}$ converge to $1$ for each $s\in\{1,\cdots,k\}$. By \Cref{Jorg}, the groups $\langle R_{s_n}, h_{t_n} \rangle$ are elementary for each $s\in\{1,\cdots,k\}$, $t\in \{1,2\}$ and sufficiently large $n$. Hence \Cref{technical} shows that $R_{s_n}=1$ for each $s\in\{1,\cdots,k\}$ and sufficiently large $n$.
\end{proof}

\begin{proof}[Proof of \Cref{dense}]
Let $G$ be dense in $\slqp$. Since $\slqp$ contains a two-generator dense subgroup, there exist $g,h\in\overline{G}$ which generate a dense subgroup $H$ of $\slqp$. Let $\{g_n\}$ and $\{h_n\}$ be sequences of elements of $G$ converging to $g$ and $h$ respectively. If all but finitely many of the subgroups $H_n =\langle g_n, h_n \rangle$ of $G$ are elementary, then $H$ is elementary by \Cref{elementary-preserve}. This is a contradiction by \cite[Lemma 6.6]{CS}, so we may assume without loss of generality that the sequence $\{H_n\}$ consists entirely of non-elementary subgroups of $G$. If all but finitely many of these subgroups are discrete, then $H$ is discrete by \Cref{algconv2}, which is again a contradiction. Hence there exists a positive integer $N$ such that $H_N$ is non-elementary and not discrete. It follows from Lemmas 6.6 and 6.7 of \cite{CS} that $H_N$ is dense.
\end{proof}

\section{Examples}

In this final section, we include several examples which are referred to throughout the paper.

\begin{exm}\label{elem-2d-exm}
Let $p$ be a prime and let $G$ be the subgroup of ${\rm SL_2}(\mathbb{F}_p((t)))$ generated by
$$X=\left\lbrace \left[ \begin{array}{cc}
1 & x \\ 
0 & 1 \end{array} \right] : x \in \mathbb{F}_p[[t]] \right\rbrace.$$
Each element of $X$ has order $p$ and fixes the vertex corresponding to $\mathbb{F}_p[[t]]^2$ in the relevant Bruhat-Tits tree. Since $X$ is infinite, $G$ is not discrete by Lemma $\ref{disc-criterion}$. However, every two elements of $X$ generate a finite (and hence discrete) group.
\end{exm}

\begin{exm}\label{alg-conv-exm}
We define the following matrices in $\slqp$:
$$A_n=\left[ \begin{array}{cc}
1 & p^n \\ 
0 & 1 \end{array} \right], \quad B=\left[ \begin{array}{cc}
p & 0 \\
1 &  \frac{1}{p} \end{array} \right], \quad C=\left[ \begin{array}{cc}
p & 0 \\
0 &  \frac{1}{p} \end{array} \right], \quad D_n=\left[ \begin{array}{cc}
1+p^n & 1 \\
0 &  \frac{1}{1+p^n} \end{array} \right].$$
Observe that:
\begin{itemize}
\item Each group $G_n=\langle A_n, B \rangle$ is non-elementary and not discrete, but $\{G_n\}$ converges algebraically to the elementary discrete group $\langle B \rangle \cong \mathbb{Z}$;
\item Each group $H_n=\langle A_n, C \rangle$ is elementary and not discrete, but $\{H_n\}$ converges algebraically to the elementary discrete group $\langle C \rangle \cong \mathbb{Z}$;
\item Each group $\langle D_n \rangle$ is elementary and discrete, but $\{ \langle D_n \rangle \}$ converges algebraically to an elementary group which is not discrete.
\end{itemize}
\end{exm}

\begin{exm}\label{tr-comm-2-ex}
Note that $-3$ is a quadratic residue modulo $7$, so Hensel's lemma shows that we can define the following matrices of ${\rm SL_2}(\mathbb{Q}_{7})$ which commute:
$$A=\left[ \begin{array}{cc}
0 & -1 \\ 
1 & 0 \end{array} \right], \quad B=\left[ \begin{array}{cc}
2 & -\sqrt{-3} \\
\sqrt{-3} &  2 \end{array} \right].$$
Since $\tr^2(A)-4=-4$ and $\tr^2(B)-4=12$ are both quadratic non-residues modulo $7$, the matrices $A$ and $B$ are not diagonalisable over $\mathbb{Q}_{7}$ and therefore fix no ends of $T_{\mathbb{Q}_7}$. 
\end{exm}

\begin{exm}\label{non-elem-iso}
Let ${\rm SL_2}(\mathbb{Z})\cong C_4 *_{C_2} C_6$ be the group generated by the matrices 
$$\left[ \begin{array}{cc}
0 & -1 \\ 
1 & 0 \end{array} \right] \text{ and } \left[ \begin{array}{cc}
0 & -1 \\ 
1 & 1 \end{array} \right];$$ see \cite[Chapter I, 1.5.3]{S}. Since ${\rm SL_2}(\mathbb{Z})$ fixes the vertex of $T_{\qp}$ corresponding to $\zp^2$ and also contains infinite order elliptic elements, it is an elementary non-discrete subgroup of $\slqp$.
On the other hand, the subgroup of $\slqp$ generated by the matrices 
$$\left[ \begin{array}{cc}
0 & -1 \\ 
1 & 0 \end{array} \right] \text{ and } \left[ \begin{array}{cc}
0 & -\frac{1}{p} \\ 
p & 1 \end{array} \right]$$ is discrete, non-elementary and also isomorphic to $C_4 *_{C_2} C_6$; for the corresponding quotient in $\pslk$, see \cite[Section 5, Case $(c)$]{CS}. 
\end{exm}

\begin{exm}\label{approx-not-fix-end}
We define the following matrices in $\slqp$:
$$
A_n=\left[ \begin{array}{cc}
1+p^n & 1 \\ 
p^n & 1 \end{array} \right], \quad
A=\left[ \begin{array}{cc}
1 & 1 \\ 
0 & 1 \end{array} \right], \quad
B=\left[ \begin{array}{cc}
p & 0 \\ 
0 & \frac{1}{p} \end{array} \right].$$
The sequence of non-discrete groups $\{\langle A_n,B\rangle\}$ converges algebraically to $\langle A,B \rangle$. Each elliptic element $A_n$ does not fix an end of $\ax(B)$, whereas $A$ does fix an end of $\ax(B)$.
\end{exm}

{\bf Acknowledgements} The first and third author are supported by the New Zealand Marsden Fund. The first author is also supported by the Rutherford Foundation. We thank Pierre-Emmanuel Caprace for pointing out the result by Breuillard and Gelander.

\end{document}